\documentclass{article}
\usepackage{amsmath}
\usepackage{amsfonts}
\usepackage{amsthm}
\usepackage{amssymb}
\usepackage[english]{babel}
\usepackage[T1]{fontenc}
\usepackage[all]{xy}

\newtheorem{theorem}{Theorem}[section]
\newtheorem{lemma}{Lemma}[section]
\newtheorem{proposition}{Proposition}[section]
\newtheorem{corollary}{Corollary}[section]
\newtheorem{notation}{Notation}[section]

\theoremstyle{definition}
\newtheorem*{definition}{Definition}
\newtheorem*{remark}{Remark}
\newtheorem*{example}{Example}

\numberwithin{equation}{section}

\begin{document}

\title{Comparison Between the Fundamental Group Scheme of a Relative Scheme and that of its Generic Fiber}


\author{\sc Marco ANTEI}

\maketitle

\textbf{Résumé.} On montre que le morphisme canonique 
$\varphi:\pi_1(X_{\eta},x_{\eta})\to \pi_1(X,x)_{\eta}$ entre
le schéma en groupes fondamental de la fibre générique  $X_{\eta}$ d'un schéma  $X$ sur un schéma de Dedekind connexe et la fibre générique du schéma en groupes fondamental de $X$ est toujours fidèlement plat. On donnera ensuite des conditions nécessaires et suffisantes pour qu'un  $G$-torseur fini, dominé et pointé au dessus de $X_{\eta}$ puisse être étendu sur $X$. On décrira des exemples où  $\varphi:\pi_1(X_{\eta},x_{\eta})\to \pi_1(X,x)_{\eta}$ est un isomorphisme.

\bigskip

\textbf{Abstract.} We show that the natural morphism
$\varphi:\pi_1(X_{\eta},x_{\eta})\to \pi_1(X,x)_{\eta}$ between
the fundamental group scheme of the generic fiber $X_{\eta}$ of  a
scheme $X$ over a connected Dedekind scheme and the generic fiber of the
fundamental group scheme of $X$ is always faithfully flat. As an application we give a necessary and sufficient condition for a finite, dominated pointed $G$-torsor over $X_{\eta}$ to be extended over $X$. We finally provide examples where $\varphi:\pi_1(X_{\eta},x_{\eta})\to \pi_1(X,x)_{\eta}$ is an isomorphism.

\bigskip
\section{Introduction}
In \cite{Sai} and \cite{MatR} (respectively)
Sa\"{i}di and Romagny give an example of a $G$-torsor $Y$ over the
generic fiber $X_{\eta}$ of a scheme $X$ over a d.v.r. (i.e. a discrete valuation ring) $R$ of equal characteristic $p>0$ whith field of fractions $K$, such that the normal closure $\overline{Y}$ of $Y$ in
$X$ does not have any structure of torsor which extends the one
given on $Y$. Namely they construct such an example when
$X=Spec(R[x])$  and $G= (\mathbb{Z}/p^2\mathbb{Z})_K$.
Nevertheless one can ask whether we can find a scheme $Y'$ and a
torsor structure on it which extends the torsor structure on
$Y$.

This problem is tightly related to the study of
the fundamental group schemes of $X$ and of $X_{\eta}$. In
\cite{No} Nori gives the definition of the fundamental group
scheme $\pi_1(X,x)$ of a reduced, connected and proper scheme $X$
over a perfect field $k$ provided with a point $x\in X(k)$.  This
definition has been extended by Gasbarri in \cite{GAS} where he replaces $k$ by a Dedekind scheme $S$ (that is to say a normal noetherian scheme of dimension $\leq 1$) and where $X$ is a reduced
and irreducible scheme faithfully flat over $S$. The two
definitions coincide if $S$ is the spectrum of a perfect field.

In the first
part of this paper we will briefly recall Nori's and Gasbarri's definitions of the fundamental group scheme and  we
state some preliminary lemmas necessary to solve our
problem. Then we study the generic fiber  of Gasbarri's fundamental group scheme of a scheme $X$ over a connected Dedekind scheme $S$ putting it in relation with the fundamental group scheme of $X_{\eta}$, the generic fiber of $X$.
 The principal results of this paper are theorems \ref{teoremaEVVAI}
and \ref{teo2}. In theorem \ref{teoremaEVVAI} we prove that the
natural morphism $\varphi:\pi_1(X_{\eta},x_{\eta})\to
\pi_1(X,x)_{\eta}$ is always faithfully flat. This has been recently  proved by Garuti (cf. \cite{GarBT} \S 4, Theorem 4) when $X$ is normal and $S$ is the spectrum of a d.v.r.. As an application we prove theorem
\ref{teo2}, that gives sufficient and necessary condition for a pointed torsor over $X_{\eta}$ to be extended over $X$. As a corollary we have that any dominated  pointed
torsor over $X_{\eta}$ (the meaning of ``dominated torsor'' will be explained in the text) can be extended to a pointed torsor over
$X$ if and only if $\varphi$ is an isomorphism. This is always the case (see proposition \ref{PropVarAb}) for $X$ an abelian scheme whose fundamental group scheme is isomorphic to the divisible group of $X$, i.e. the inverse limit of the kernels of the multiplication by an integer maps (cf. again proposition \ref{PropVarAb}) using what Nori has already proved over a field (cf. \cite{No3}).  In this case we  prove that any pointed torsor (not only dominated) over the generic fiber $X_{\eta}$ of $X$ can be extended if and only if the finite group scheme acting on it has a model, which is always true when the field of functions of $S$ has characteristic $0$ (cf. \ref{PropVarAb2}). 

\section{The fundamental group scheme}
\subsection{Preliminaries}\label{sez:preliminari}
In \cite{No}, Nori defines the fundamental group scheme
$\pi_1(X,x)$ of a reduced, connected and proper scheme $X$ over a
perfect field $k$ provided with a point $x\in X(k)$ as the group
scheme associated to the neutral tannakian category
$(EF(X),\otimes,x^{\ast},\mathcal{O}_X)$ of essentially finite vector bundles over $X$. In \cite{No2}, Part I, Ch. II, \S 1
Nori gives a second equivalent description for his
fundamental group scheme. Gasbarri in \cite{GAS} develops this point of view. We give some details on Gasbarri's
construction. So
from now on  let $S$ be a Dedekind scheme, $X$ a reduced,
irreducible scheme, $j:X\rightarrow
S$ a faithfully flat morphism and $x:S\rightarrow X$ a fixed $S$-valued point. 

\begin{definition}\label{defTRIPLETS}
Let $\mathcal{P}(X)$ be the category whose objects are triples
 $(Y,G,y)$ where:
\begin{enumerate}
  \item $G$ is a finite and flat $S$-group scheme.
  \item $f:Y\rightarrow X$ is a $G$-torsor for the \textit{fpqc} topology.
  \item $y:S\rightarrow Y$ is a section such that $f(y)=x$.
\end{enumerate}
A morphism $\varphi:(Y_1,G_1,y_1)\rightarrow (Y_2,G_2,y_2)$
between two triples is the datum of two morphisms
$\alpha:Y_1\rightarrow Y_2$ and $\beta: G_1\rightarrow G_2$ where
$\beta$ is a group scheme morphism, $\alpha(y_1)=y_2$ and such that the
following diagram

\begin{displaymath}
\begin{array}{ccc}
  G_1\times Y_1 & \rightarrow & Y_1 \\
  \downarrow & \circlearrowleft & \downarrow \\
  G_2\times Y_2 & \rightarrow & Y_2
\end{array}
\end{displaymath}
commutes, the horizontal arrows being the actions.
\end{definition}

\begin{lemma}\label{lemmadelTERZOCAP} Let $\alpha:G\to H$ be a group scheme morphism,
$Y$ a $G$-torsor over $X$, $P$ an $H$-torsor over $X$ and
$\varphi:Y\to P$ a morphism between torsors compatible with the
actions of $G$ and $H$. Then $P$ is isomorphic to the contracted product $Y\times^G H$ (as defined in \cite{DG},
III, \S 4, 3.2).\end{lemma}
\proof For any $X$-scheme $T$
we have a canonical arrow:
$$\begin{array}{ccc}
Y(T)\times H(T)& \to & P(T)\\
(y,h)&\mapsto & \varphi(y)\cdot h
\end{array}$$
that passes to quotient (under the left action of $G$). We deduce
a morphism of $H$-torsors $Y\times^G H\to P$ over
$X$ which is then an isomorphism since every morphism between $H$-torsors is an isomorphism, hence the desired result.\endproof

Let $I$ be the set of isomorphism classes of objects of $\mathcal{P}(X)$; it is a poset when
provided with the following  relation: if $i,j\in I$ then $i\leq
j$ if and only if there exists a  morphism from the triple corresponding to $j$ to the triple corresponding to $i$. Moreover the following theorem holds:

\begin{theorem}\label{teoGABARRI} The set
$I$ is filtered. Then we can define a
pro-object\\
$\underleftarrow{lim}_{i\in
I } (Y_i,G_i,x_i)$. Moreover, $\pi_1(X,x)=\underleftarrow{lim}_{i\in
I }G_i$ is an $S$-group scheme
and $\widetilde{Y}=\underleftarrow{lim}_{i\in
I }Y_i$ is a scheme.\end{theorem} \proof 
See \cite{GAS},
Proposition 2.1.\endproof

\begin{definition}\label{defGRUPPOGASBARRI}  We call the $S$-group scheme $\pi_1(X,x)$
 constructed in theorem
\ref{teoGABARRI} the fundamental group scheme. We call
the scheme  $\widetilde{Y}$ the $\pi_1(X,x)$-universal torsor
over $X$.\end{definition}

\begin{remark}\label{ossGRUPPOGASBARRI} When they can be compared Nori's and Gasbarri's construction coincide (cf. \cite{No2}, Ch. II). Thus  from now on we denote by $\pi_1(X,x)$
both Nori's and Gasbarri's fundamental group schemes and no confusion will arise.
\end{remark}

\begin{remark}\label{ossTRIPLEeCOPPIE} There is a bijection between isomorphism classes of triples $(Y,G,y)$  as in def. \ref{defTRIPLETS}  and $S$-group scheme morphisms $\rho:\pi_{1}(X,x)\rightarrow G$. Indeed given a triple $(Y,G,y)$, the morphism $\rho:\pi_{1}(X,x)\rightarrow G$ comes directly from 
theorem \ref{teoGABARRI}. On the other direction it is
sufficient to consider the contracted product
$\widetilde{Y}\times^{\pi_{{1}}(X,x)}G$. Lemma \ref{lemmadelTERZOCAP} ensures that one direction is the
inverse of the other.\end{remark}

\subsection{Some elementary lemmas}\label{sez:lemmiHOPF}

If not stated otherwise $S$ is any scheme. All group schemes that we will consider will be affine over $S$. We recall that a morphism of schemes $f:Z\to Y$ is called schematically dominant (cf. \cite{EGAIV-3}, Définition 11.10.2) if the corresponding morphism of $\mathcal{O}_X$-modules $f^{\#}:\mathcal{O}_Y\to f_{\ast}(\mathcal{O}_Z)$ is injective.

\begin{remark}\label{ossQUOZWW} Let $\beta:G'\rightarrow G$ be a morphism of affine (not necessarily finite) group schemes over a
field. The following are equivalent:
\begin{enumerate}
  \item $\beta:G'\rightarrow G$ is schematically dominant,
  \item $\beta:G'\rightarrow G$ is surjective for
the  \textit{fpqc} topology,
  \item $\beta:G'\rightarrow G$ is faithfully
flat.
\end{enumerate}
Indeed 1) $\Leftrightarrow$ 2) comes from  \cite{WW}, Ch. 15, \S 5 and  2) $\Leftrightarrow$ 3)  \cite{WW}, Ch. 14, \S 1.
\end{remark}

\begin{definition}\label{defTRIPLERIDOTTE}Let $S$ be a Dedekind scheme. A triple $(Y,G,y)$, as defined before,  is said to be a
dominated triple\footnote{N.B.: such a triple is called a
``reduced triple" in  \cite{No2}, Part I, Ch. II when $S$ is the spectrum of a field. Because of the
confusion that can arise we have decided to call it in a
different manner.} if for any triple $(Y',G',y')$ and any morphism
$\varphi=(\alpha,\beta):(Y',G',y')\rightarrow (Y,G,y)$, $\beta$
is a schematically dominant morphism. We will often refer to a triple (resp. dominated triple) $(Y,G,y)$ as a pointed (resp. dominated pointed) torsor.
\end{definition}

For the sake of completeness we give some details for the easy proofs of the following useful lemmas.

\begin{lemma}\label{lemFATTORIZZAGRUPPI}
Any $S$-morphism $f:G'\rightarrow G$ between group schemes  can be factored into a schematically dominant morphism
$s:G'\to F$ ($F$ some $S$-group scheme) and a
closed immersion  $i:F\hookrightarrow G$ such that $i\circ s = f$:
$$\xymatrix{G' \ar[rr]^{f} \ar[dr]_{s}& & G  \\
& F \ar@{^{(}->}[ur]_{i} } $$\end{lemma}\proof The question being local on $S$, we can assume $S$ to be an affine scheme, thus we set  $S:=Spec(B)$, $G:=Spec(C)$, $G':=Spec(A)$ to be affine. Denote by $h:A\to C$ the $B$-Hopf algebra morphism corresponding to $f$. Then $F:=Spec(Im(h))$ is
the $B$-group scheme with the desired properties. Indeed let $\Delta_A:A\rightarrow A\otimes A$ and $\Delta_C: C\rightarrow C\otimes C$ be
the comultiplications of (resp.) $A$ and $C$. By the following  commutative diagram:
$$\xymatrix{A \ar[r]^{h} \ar[d]_{\Delta_A} & C \ar[d]^{\Delta_C}\\ A\otimes A \ar[r]_{h\otimes h}  & C\otimes C  }$$

\noindent one deduces that $H:=Im(h)$, the $B$-submodule image of $h$ (then $B$-flat as $C$ is), is a $B$-subcoalgebra of $C$ since
$\Delta_C(H)\subseteq H\otimes H$. 
Similarly let $m_A:A\otimes A\rightarrow A$ and
$m_C:C\otimes C\rightarrow C $ be the multiplications of  $A$ and $C$ we have 

$m_C(H\otimes H)\subseteq H$. If moreover $S_A$ and $S_C$ are the antipodal morphisms of $A$ and $C$
one can easily verify that $S_C(H)\subseteq H$ and this gives $H$ the desired $B$-Hopf algebra structure.
\endproof

\begin{lemma}\label{lem2.2} Let  $(G_i,\gamma_i^l)_{i\in I}$ be an inverse system of $S$-group schemes
and $G= \underleftarrow{lim}_{i\in I}G_i$; we have, for any pair
$(i,l)$ such that $i\leq l$, the following commutative diagram:
$$\xymatrix{G_i   & \ar[l]_{\rho_i} G \ar[dl]^{\rho_l}\\ G_l  \ar[u]^{\gamma_i^l}.  &  }$$
Then the canonical morphism $\rho_i$  is schematically dominant if and only if for any
$l\geq i$ the map $\gamma_i^l:G_l\rightarrow G_i$ is
schematically dominant.
\end{lemma}

\begin{proof} As before we set $S:=Spec(B)$, $G_i:=Spec(A_i)$ and $G:=Spec(A)=Spec(\underrightarrow{lim}_{i\in I}A_i)$ and we prove the dual statement for $B$-Hopf algebras, where $f_i^l:A_i\rightarrow A_l$ corresponds to $\gamma_i^l$ and $\alpha_i:A_i\to A$ to $\rho_i$. One direction is obvious. In the
other direction, we suppose that $f_i^l$ is
injective for all $l\geq i$; let $x\in A_i$ and $\alpha_i(x)=0$.
Now, we set $y:=f_i^l(x)\in A_l$, we know that $\alpha_l(y)=0$
according to the previous diagram; but $\alpha_l$ is defined as
the composition of the following morphisms:
\begin{displaymath}
\begin{array}{cccccc}
  \alpha_l:& A_l &\hookrightarrow & \coprod_{k \in I} A_k
  &\twoheadrightarrow &
\frac{\coprod_{k\in I} A_k}{\thicksim}\simeq A \\
   & y & \mapsto & y & \mapsto & 0
\end{array}
\end{displaymath}
where, for $a_i\in A_i$ and  $a_j\in A_j$, we write $a_i \thicksim a_j$ if and only if there exists $k\geq i$,
$k\geq j$ such that $f_i^k(a_i)=f_j^k(a_j)$ and this means that there
exist $r\in I$, $r\geq l$ and $f_l^r:A_l\rightarrow A_r$ such that
$f_l^r(y)=0$, in particular the morphism $f_i^r=f_l^r\circ
f_i^l:A_i\rightarrow A_r$  maps $x$ into $0$, but according to the assumption on $A_i$ the morphism
$f_i^r$ is injective and then
$x=0$.
\end{proof}

As a consequence we have the following

\begin{corollary}\label{corollRIDOTTO}Let $S$ be a Dedekind scheme. A triple $(Y,G,y)$ is
dominated  if and only if the
morphism $\rho:\pi_{1}(X,x)\rightarrow G$ naturally associated to
this triple  is a schematically dominant
morphism.\end{corollary}

\begin{lemma}\label{lemmaSCIAPO} Let $S$ be a Dedekind scheme. Let $\rho:\pi_{1}(X,x)\rightarrow G$ be an $S$-morphism of group schemes where $G$ is finite and flat over $S$. Then there exist an $S$-group scheme $G'$ finite and flat over $S$, an $S$-morphism of group schemes
$\rho':\pi_{1}(X,x)\to G'$ and a closed immersion $\beta:G' \to  G $ such that
$\beta\circ \rho'=\rho$ and $\rho'$ is a schematically dominant  morphism.\end{lemma}\proof The
existence of morphisms $\rho':\pi_{1}(X,x)\to G'$
(schematically dominant) and $\beta:G' \rightarrow  G$ (closed immersion) such that $\beta\circ
\rho'=\rho$
$$\xymatrix{\pi_{1}(X,x) \ar[rr]^{\rho} \ar[dr]_{\rho'}& & G  \\
& G' \ar@{^{(}->}[ur]_{\beta} } $$ is ensured by  lemma
\ref{lemFATTORIZZAGRUPPI}.  The pointed torsor associated to $\rho'$ is then
dominated.\endproof

According to lemma \ref{lemmaSCIAPO} we can say that any triple  $(Y,G,y)$ is 
\textit{preceded} by a dominated triple or equivalently any pointed torsor over $X$ is \textit{preceded} by a dominated pointed torsor.

\begin{lemma}\label{lem2.3} Let 
  $(G_i,\gamma_i^j)_{i\in I}$ be an inverse system of $S$-group schemes, $G= \underleftarrow{lim}_{i\in I}G_i$ and
$\rho_i: G\rightarrow G_i$ the canonical morphism. Let $J$ be a filtered subset of $I$
and  $G'= \underleftarrow{lim}_{j\in J}G_j$. We assume that for
any $j\in J$ the canonical morphism
$$\rho'_j:G'\rightarrow G_j$$ is schematically dominant. Then the natural morphism  $\varphi:G\rightarrow G'$ is schematically dominant if
and only if $\rho_j:G\rightarrow G_j$ is schematically dominant for any
$j\in J$.\end{lemma}

\proof Again we can assume that $S$ is affine, so we set $S:=Spec(B)$, $G_i:=Spec(A_i)$, $G:=Spec(A)=Spec(\underrightarrow{lim}_{i\in I}A_i)$ and $G':=Spec(C)=Spec(\underrightarrow{lim}_{j\in J}A_j)$. For any $j\in J$
we have the following commutative diagram:

$$\xymatrix{C\ar[r]^{\psi} & A\\ A_j \ar@{^{(}->}[u]^{\gamma_j} \ar[ur]_{\alpha_j}}$$
where $\psi$, $\alpha_j$ and $\gamma_j$  correspond respectively to $\varphi$, $\rho_j$ and $\rho'_j$. If $\psi$ is injective then
$\alpha_j$ is injective too for all $ j\in J$ (obvious). Conversely, suppose $\alpha_j$ injective for all $j\in J$ and
let $x\in C$ be such that   $\psi (x)=0$. Once again we make use
of the canonical factorisation:
\begin{displaymath}
\gamma_j: A_j \hookrightarrow  \coprod_{u \in J} A_{u}
  \twoheadrightarrow
\frac{\coprod_{u \in J} A_{u}}{\thicksim}\simeq C .
\end{displaymath}
So let $z\in \coprod_{u \in J} A_{u}$ be a representing element of
$x\in C$, it follows that there exists $v\in J$ such that $z\in A_v$
and $\gamma_v(z)=x$, in particular we have  $0=\psi(x)=\psi \circ
\gamma_v (z)= \alpha_v(z)$, but since we have assumed $\alpha_v$
to be injective we have  $z=0$, then $x=\gamma_v(z)=0$.\endproof

\begin{lemma}\label{lemmaSURIETT} Let 
  $(G_i,\gamma_i^j)_{i\in I}$ be an inverse system of $S$-group schemes, $G= \underleftarrow{lim}_{i\in I}G_i$ and
$\rho_i: G\rightarrow G_i$ the canonical morphism. Let $J$ be a filtered subset of $I$
and  $G'= \underleftarrow{lim}_{j\in J}G_j$. If for every $i\in I$ there exists a morphism $\gamma_i:G'\to G_i$ such that the following diagram 
$$\xymatrix{G\ar[r]^{\varphi}\ar[d]_{\rho_i} & G'\ar[dl]^{\gamma_i}\\ G_i  & }$$
commutes, then $\varphi$ is a closed immersion. If moreover every $\rho_i$ is schematically dominant then $\varphi$ is an isomorphism.  \end{lemma}\proof Similar to the previous one.\endproof

\subsection{A comparison theorem}\label{sez:structure}

From now on  $S$  will denote  a connected  Dedekind
scheme. We have recalled that the fundamental group scheme is the projective limit of $S$-finite and flat group schemes as follows
$$\pi_{1}(X,x):=\underleftarrow{lim}_{i\in I} G_i;$$ 
we will denote by
$$\rho_i:\pi_{1}(X,x)\rightarrow G_i$$
the corresponding canonical morphisms.

\begin{proposition}\label{propGRUPPOFONDLINIDIRIDOT}
Let $J\subseteq I$ be the set of all $i\in I$ such that
$\rho_i:\pi_{1}(X,x)\to G_i$ is a schematically dominant morphism. The group scheme $\pi_{1}(X,x)$ is isomorphic to the
projective limit of all the finite and flat $S$-group schemes $G_j$,
$j\in J$, i.e. $\pi_{1}(X,x)\simeq \underleftarrow{lim}_{j\in J} G_j$.
\end{proposition}
\proof $J$ is filtered, indeed the construction made in \cite{GAS}, Proposition 2.1 still holds: so given two dominated pointed torsors $(Y_1,G_1,y_1)$ and $(Y_2,G_2,y_2)$ and morphisms over a third dominated pointed torsor $(Y,G,y)$ over $X$ we can construct a fourth pointed torsor $(P,H,p)$ with morphisms over the first two making the diagram $$\xymatrix{(P,H,p)\ar[d]\ar[r] & (Y_1,G_1,y_1)\ar[d]\\(Y_2,G_2,y_2)\ar[r]  & (Y,G,y)}$$ commute. If $(P,H,p)$ is not a dominated torsor it is sufficient to use lemma \ref{lemmaSCIAPO} in order to find a dominated pointed torsor
$(P',H',p')$ making a similar diagram commute. Since $J\subset I$ we have a canonical morphism $v:\pi_1(X,x)\to \underleftarrow{lim}_{j\in J} G_j$ which is an isomorphism: indeed it is schematically dominant by lemma \ref{lem2.3}, since by definition for all $j\in J$ $\xymatrix{\pi_1(X,x) \ar[r]^{v} & \underleftarrow{lim}_{j\in J} G_j \ar[r]& G_j}$ is schematically dominant, and a closed immersion by lemma \ref{lemmaSURIETT} since for all $i\in I$ there exists $j\in J$ such that the canonical morphism $\rho_i:\pi_1(X,x)\to G_i$ factors through $G_j$ hence we have a natural morphism $v_i:\underleftarrow{lim}_{j\in J} G_j\to G_j \to G_i$ such that $v_i\circ v= \rho_i$.\endproof

Now, let ${\eta}$ be the generic point of $S$, we
construct $X_{\eta}:=X\times_{S}{\eta}$ that possesses a point
$x_{\eta}\in X_{\eta}({\eta})$ (fiber of $x\in X(S)$). Over
${\eta}$ we can construct the fundamental group scheme
$$\pi_{1}(X_{\eta},x_{\eta}):=\underleftarrow{lim}_{m\in M} F_m,$$  where $M$ is the set of isomorphism classes of objects of $\mathcal{P}(X_{\eta})$, but also the group scheme
$$\pi_{1}(X,x)_{\eta}:=\pi_{1}(X,x)\times_S {\eta}.$$ 

Now we compare the fundamental group scheme
$\pi_{1}(X_{\eta},x_{\eta})$ and the generic fiber
$\pi_{1}(X,x)_{\eta}$ of the fundamental group scheme of $X$.

Direct limits of algebras commute with base change (cf.
\cite{Mats}, Appendix A, Theorem A.1), so the same is true for
inverse limits of affine group schemes. Hence one gets

$$\pi_{1}(X,x)_{\eta}\simeq\underleftarrow{lim}_{j\in J} G_{j,\eta}$$
where $G_{j,\eta}:=
G_j\times_S{{\eta}}$ and $\rho_{j,\eta}:\pi_{1}(X,x)_{\eta}\rightarrow G_{j,\eta}$ denotes the fiber of $\rho_j$. Since $\eta\to S$ is flat and $\rho_j$ is schematically dominant then $\rho_{j,\eta}$ is schematically dominant too. We consider the set $J':=J/\sim$ where for $j_1,j_2\in J$ we set $j_1\sim j_2$ if and only if $\rho_{j_1,\eta}\simeq \rho_{j_2,\eta}$; thus $$(\dagger)\qquad\qquad\qquad\qquad\qquad \pi_{1}(X,x)_{\eta}\simeq\underleftarrow{lim}_{j\in J'} G_{j,\eta}.\qquad\qquad\qquad\qquad\qquad$$

\begin{lemma}\label{lemmaREFEREE} $J'$ is  filtered and there is an injective map $J' \hookrightarrow M$.\end{lemma} \proof We are given two pointed torsors $(Y_{1,\eta},G_{j_1,\eta},y_{1,\eta})$ and $(Y_{2,\eta},$ $G_{j_1,\eta},y_{2,\eta})$ and morphisms over a third pointed torsor $(Y_{\eta},G_{j,\eta},y_{\eta})$ over $X_{\eta}$ (for $j,j_1,j_2\in J'$) then since $J\subset I$ is filtered there exists a triple $(Y',G',y')$ over $X$ dominating $(Y_1,G_1,y_1)$ and $(Y_2,G_2,y_2)$ whose generic fiber makes the obvious diagram commute.\endproof

As a consequence there exists a
morphism

$$\varphi:\pi_{1}(X_{\eta},x_{\eta})\longrightarrow \pi_{1}(X,x)_{\eta}.$$

We can now state the principal result of this section:

\begin{theorem}\label{teoremaEVVAI}
The morphism $\varphi:\pi_{1}(X_{\eta},x_{\eta})\longrightarrow
\pi_{1}(X,x)_{\eta}$ is faithfully flat.\end{theorem}

As a first application of this result we give a non trivial example where the fundamental group scheme is trivial. Thus in particular  $\pi_1(\mathbb{P}^1_S,x)$ is trivial for
$S$ a connected Dedekind scheme and $x\in \mathbb{P}^1_S(S)$: 

\begin{example}\label{esempipPI1}
Let $S$ be a connected
Dedekind scheme, $X$ an integral scheme, faithfully 
flat over $S$ and $x : S \to X$ a section.
If the generic fiber of $X$ is a complete normal rational variety, then $\pi_1(X,x)$ is trivial\end{example}\proof Let $\eta$
be the generic point of $S$. That
$\pi_1(X_{\eta},x_{\eta})$ is trivial follows by \cite{No2}, Ch II, Proposition 9 and its corollary). The fundamental group scheme $\pi_1(X,x)$ is flat over $S$ and has trivial generic fiber then it coincides with the scheme theoretic closure of $\{1\}_{\eta}$ in $\pi_1(X,x)$, which is then trivial. 
\endproof

\begin{proof} [Proof of Theorem] \ref{teoremaEVVAI}. We prove that 
\begin{displaymath}
\varphi:\pi_{1}(X_{\eta},x_{\eta})\longrightarrow
\pi_{1}(X,x)_{\eta}
\end{displaymath}
is a schematically dominant morphism. For any $j\in J'$ we consider the
following commutative diagram:
$$\xymatrix{\pi_1(X_{\eta},x_{\eta}) \ar[r]^{\varphi}  \ar[dr]_{q_j} & \pi_1(X,x)_{\eta} \ar[d]^{\rho_{j,\eta}}\\
  & G_{j,\eta}   } $$

\noindent According to lemma \ref{lem2.3} 
since $\rho_{j,\eta}:\pi_1(X,x)_{\eta}\to
G_{j,\eta}$ is schematically dominant for any $j\in J'$ (which is filtered in $M$, cf. lemma \ref{lemmaREFEREE}), it is sufficient to prove that for
all $j\in J'$ the morphism
$q_j:\pi_1(X_{\eta},x_{\eta})\rightarrow G_{j,\eta}$ is schematically dominant
too.  By lemma \ref{lemFATTORIZZAGRUPPI} we split
$q_j:\pi_1(X_{\eta},x_{\eta})\rightarrow G_{j,\eta}$ into a schematically dominant morphism followed by a closed immersion:

$$
q_j:\pi_1(X_{\eta},x_{\eta}) \to G \hookrightarrow
G_{j,\eta}.
$$
where $G$ is a finite $K$-group scheme. 

Let $(Y',G,y')$ be the dominated pointed torsor associated  to $\pi_1(X_{\eta},x_{\eta})\to G$ and let $(Y,G_{j,\eta},y)$ be the  pointed torsor associated to $q_j:\pi_1(X_{\eta},x_{\eta})\to G_{j,\eta}$. The latter is isomorphic to the contracted product (cf. lemma
\ref{lemmadelTERZOCAP}) $$Y\simeq Y'\times^{G}G_{j,\eta}$$
 via the morphism
$f:G\hookrightarrow G_{j,\eta}$ , and $y$ is the image in $Y$ of $y'$.
We immediately observe that the canonical morphism  $f':Y'\rightarrow Y$ is a closed
immersion. Indeed locally, for the
\textit{fpqc} topology, it is certainly true since locally any
torsor is trivial. We deduce
that the result
is also true globally since being a closed immersion is a local
property for the
\textit{fpqc} topology (\cite{EGAIV-2}, Proposition 2.7.1.).
By construction there exist a finite and flat $S$-group scheme
$G_j$ and a dominated triple $(P,G_j,p)$ over $X$  such that
$(Y,G_{j,\eta},y)$ is its generic fiber. The following diagram
describes the present situation:
$$\xymatrix{
  (Y',G,y')\ar[d]   &  \\   
  (Y,G_{j,\eta},y)\ar[d]\ar[r]   & (P , G_j, p)\ar[d] \\
    X_{\eta} \ar[r]\ar[d]& X\ar[d]   \\   
    {\eta}\ar[r]  & S 
}
$$

According to \cite{EGAIV-2}, proposition 2.8.5, there is a unique
 $S$-group scheme $H$, closed  subgroup scheme
of $G_j$ which is flat over   $S$ and such that
$H\times_{S}{\eta}\simeq G$: it's the scheme theoretic closure
of $G$ in $G_j$. Similarly we construct   $Q$, the only closed
subscheme of $P$ which is flat over   $S$ and such that
$Q\times_{S}{\eta}\simeq  Y'$. Again we construct the section $q:S\rightarrow Q$ as the  scheme theoretic closure of $y'$ in $p$. We have the following

\begin{lemma}\label{lemmaLACHIUSURATRSORE} $(Q,H,q)$ is a pointed torsor over $S$.\end{lemma}  \proof  (this is lemma 2.2 of  \cite{GAS} whose proof will be sketched here for the comfort of the reader)
  The scheme theoretic closure of closed subschemes of the generic fibre is fonctorial and
commutes with fiber products (cf. \cite{EGAIV-2}, (2.8.3) and Corollaire
2.8.6) so in particular from diagram
$$\xymatrix{ G\times Y' \ar[d] \ar[rr]^{\qquad action}& & Y'\ar[d]\\ G_{j,\eta}\times Y \ar[rr]^{\qquad action}& & Y} $$
we deduce the commutative diagram
$$\xymatrix{ H\times Q \ar[d] \ar[rr]^{\qquad action}& & Q\ar[d]\\ G_j\times P \ar[rr]^{\qquad action}& & P} $$
hence an action $H\times Q\to Q$. The isomorphism  $G\times
Y'\simeq Y'\times_{X_{\eta}} Y'$ implies the isomorphism
$\overline{G}\times \overline{Y'}\simeq
\overline{Y'\times_{X_{\eta}} Y'}$; the latter is the only closed subscheme of  $P\times_{X} P$, flat over $S$ and whose generic fiber is isomorphic to $Y'\times_{X_{\eta}} Y'$. Since the same properties are satisfied by $Q\times_{X} Q$ then  $\overline{Y'\times_{X_{\eta}} Y'} \simeq Q\times_{X} Q$ and consequently
 $H\times Q\simeq Q\times_{X} Q$ and $Q$ is a $H$-torsor over $X$. With similar remarks we get the desired pointed torsor considering the section $q:S\rightarrow Q$.\endproof

Thus we have a commutative diagram of triples:

$$\xymatrix{
  (Y',G,y')\ar[d]\ar[r]   & (Q,H,q)\ar[d] \\   
  (Y',G_{j,\eta},y')\ar[d]\ar[r]   & (P , G_j, p)\ar[d] \\
    X_{\eta} \ar[r]\ar[d]& X\ar[d]   \\   
    {\eta}\ar[r]  & S 
}
$$
\newline
where the morphism $H\to G_j$ is by construction a closed immersion. But $(P,G_j,p)$ is a dominated pointed triple hence by definition the morphism $H\to G_j$ is also schematically dominant then an isomorphism. So the same is true for the morphism $G\to G_{j,\eta}$, which proves that  $q_j:\pi_{1}(X_{\eta},x_{\eta})\to
G_{j,\eta}$ was already schematically dominant and this is enough to conclude.\end{proof}

Proposition \ref{PropVarAb} will provide an example where the morphism $$\varphi:\pi_{1}(X_{\eta},x_{\eta})\to
\pi_{1}(X,x)_{\eta}$$ is actually an isomorphism. 

\section{Applications}
\subsection{Extension of torsors}
Now we  apply theorem \ref{teoremaEVVAI} to the problem of
extending  torsors or, more precisely, we explain
how the kernel  $N$ of the morphism
$\varphi:\pi_{1}(X_{\eta},x_{\eta})\to \pi_{1}(X,x)_{\eta}$
measures the obstruction to extending a torsor over $X_{\eta}$
under the action of a finite group scheme to a
torsor over  $X$ under the action of a finite and flat $S$-group
scheme; we fix some notations for this section:

\begin{notation}\label{notazIPOTESIFINALI}

From now on $S$ will be a connected Dedekind scheme, $\eta:=Spec(K)$ its
function field, $X$ an integral scheme and $j:X\rightarrow S$ a faithfully flat morphism. We
fix a section $x:S\rightarrow X$. We set
$N:=ker(\varphi)$ where
$\varphi:\pi_{1}(X_{\eta},x_{\eta})\longrightarrow
\pi_{1}(X,x)_{\eta}$ is the canonical morphism already described.
\end{notation}

We prove the following 

\begin{theorem}\label{teo2}
Let  $G$ be a finite group scheme over $K$,
$\rho:\pi_{1}(X_{\eta},x_{\eta})\to G$ a  schematically dominant
morphism of $K$-group
schemes and $(Y,G,y)$ the associated dominated
pointed torsor. Then there exists a pointed torsor  $(Y',G',y')$ over $X$, with $G'$ a finite 
flat $S$-group scheme, whose generic fiber is isomorphic to $(Y,G,y)$ if
and only if
$N<ker(\rho)$.
\end{theorem}

A consequence of theorem \ref{teo2} (and lemma \ref{lemmaSURIETT}) is the following

\begin{corollary}\label{corollBELLOeFINALE} Any dominated triple over $X_{\eta}$ can be extended to a
(dominated) triple over $X$ if and only if
$\varphi:\pi_{1}(X_{\eta},x_{\eta})\to
\pi_{1}(X,x)_{\eta}$ is an isomorphism.\end{corollary}

\begin{proof} [Proof of Theorem] \ref{teo2}. One direction  is simple: assume in fact that there exists a triple
$(Y',G',y')$ over $X$ whose generic fiber is isomorphic
to $(Y,G,y)$. This means that there exists a morphism
$\rho':\pi_{1}(X,x)\rightarrow G'$ whose generic fiber
$\rho'_{\eta}:\pi_{1}(X,x)_{\eta}\rightarrow G'\times_S
\eta\simeq G$ satisfies $\rho'_{\eta}\circ \varphi =\rho $,  that is
the following diagram commutes:

$$\xymatrix{\pi_{1}(X_{\eta},x_{\eta})\ar[r]^{\varphi} \ar[dr]_{\rho} & \pi_{1}(X,x)_{\eta} \ar[d]^{\rho'_{\eta}} \ar[r]&
\pi_1(X,x)\ar[d]^{\rho'}\\
& G \ar[r] & G'\\ } $$ \noindent The existence of   $\rho'_{\eta}$ is equivalent (cf. \cite{WW}, Ch. 15, Theorem 15.4) to the condition $N<ker(\rho)$. Now, suppose that the
condition  $N<ker(\rho)$ holds, then there exists a schematically dominant morphism
$\gamma:\pi_{1}(X,x)_{\eta}\to G$ such that
$\gamma\circ \varphi = \rho $, that is the following diagram
commutes:

$$\xymatrix{\pi_{1}(X_{\eta},x_{\eta})\ar[r]^{\varphi} \ar[dr]_{\rho} & \pi_{1}(X,x)_{\eta} \ar[d]^{\gamma} \\
 & G  \\ } $$
\noindent   We recall that $\pi_{1}(X,x)_{\eta}\simeq \underleftarrow{lim}_{j\in J'}
G_{j,\eta}$ where
$\rho_{j,\eta}:\pi_{1}(X,x)_{\eta}\to G_{j,\eta}$ is schematically dominant for all $j\in J'$ (cf. isomorphism $(\dagger)$). Since to quotient
$\pi_{1}(X,x)_{\eta}$, which is often not of finite type,  by  $N$
 could be a problem we first need the following 

\begin{lemma}\label{lemma2punto4}
There exists $j\in J'$ such that $\gamma$ factors
through $G_{j,\eta}$, i.e. there exists a morphism $\gamma_j: {G_j}_{\eta}  \to G$ such that the following diagram

$$\xymatrix{\pi_{1}(X,x)_{\eta}\ar[r]^{{\rho_j}_{\eta}} \ar[d]_{\gamma} & {G_j}_{\eta} \ar[dl]^{\gamma_j} \\ G &
}$$ commutes. \end{lemma}
\proof 
This follows directly from the finiteness of $G$.
\endproof

If $G$ is any $S$-group scheme and $H$ a closed subgroup scheme of $G$ we denote by $G/H_{(fpqc)}$  the sheaf associated, with respect to the  fpqc topology, to the functor $$T\mapsto G(T)/H(T)$$ from the category of schemes over $S$ to the category of sets. If $G/H_{(fpqc)}$ is represented by a $S$-scheme we denote it by $G/H$.

\begin{proposition}\label{propRUBATAadSGA}
Let $G$ and $H$ be two $S$-group schemes (here $S$ need not be a Dedekind scheme), $H\hookrightarrow G$ a closed immersion and assume that  $G/H_{(fpqc)}$ is represented by a scheme $G/H$, then:

\begin{enumerate}
  \item if $H$ is a normal closed subgroup scheme of
   $G$ then on $G/H$ there exists a unique structure of $S$-group scheme such that the canonical morphism $p:G\to G/H$ is a morphism of
   $S$-group schemes.
  \item Let $T$ be any $S$-scheme and set $G_T:=G\times_S T$
  and $H_T:=H\times_S T$. Then ${G_T/H_T}_{(fpqc)}$ is represented by the
  $T$-scheme  $(G/H)\times_S T$.
  \item  the canonical morphism $p:G\to G/H$ is faithfully flat if and only if $H$ is flat over $S$.
\end{enumerate}\end{proposition}
\proof Cf.  \cite{BER},
 Proposition 9.2 (resp.) (iv), (v) and (xi).\endproof

Now we recall a particular case of \cite{Gab}, Th\'eor\`eme 7.1  which fits to our situation since finite implies projective (cf. \cite{EGAII} Corollaire 6.1.11):

\begin{theorem}\label{teoRappre1}
Let $S$ be any connected scheme. Let $G$ be a $S$-group scheme of finite type and let $H$ be a closed subgroup scheme of $G$, proper and flat over $S$. If $G$ is quasi projective over $S$ then  $G/H_{(fpqc)}$ is representable.
\end{theorem}

Now we come back to the proof of theorem \ref{teo2}. The morphism
$\gamma_j:G_{j,\eta}\rightarrow G$ from lemma \ref{lemma2punto4} is schematically dominant since $\gamma$ is.
We set
\begin{displaymath}
N_1:=ker(\gamma_j)
\end{displaymath}
which is a closed subgroup scheme of $N_1$. According to
\cite{EGAIV-2}, Proposition 2.8.5, we construct the scheme
theoretic closure of $N_1$ in $G_j$, that is an $S$-scheme $N_2$
which is the only closed subgroup scheme of $G_j$ flat over $S$
whose fibre is isomorphic to $G_{j,\eta}$. Moreover, according to \cite{ANA},
remarque 1.2.5. $d)$, $N_2$ is normal in $G_j$. Let us denote by $G'$ the $S$-quotient scheme 
$G_j/N_2$. Moreover, according to proposition \ref{propRUBATAadSGA},
ii) there is an isomorphism $G\simeq G_{j,\eta}/N_1\simeq
G'\times_{S}{\eta}$.  Then we have the
following commutative diagram:

 $$\xymatrix{N_1\ar[r] \ar@{^{(}->}[d]& N_2 \ar@{^{(}->}[d]\\
 {G_j}_{\eta}\ar[d]_{\gamma_j} \ar[r] & {G_j}\ar[d]^{\gamma_j'}\\
 G\ar[r]\ar[d] & G'\ar[d]\\
 {\eta}\ar[r] & S}$$
\noindent where we have denoted by $\gamma_j'$ the morphism $G_j
\rightarrow G'$; we compose it with
$\rho_j:\pi_{1}(X,x)\rightarrow G_j$ in order to obtain a
morphism $\gamma_j'\circ\rho_j:\pi_{1}(X,x)\rightarrow G'$ to
which we associate the triple $(Y',G',y')$ (cf. rem.
\ref{ossTRIPLEeCOPPIE}) which is the desired triple. This
concludes the proof of theorem \ref{teo2}.\end{proof}

Now we explain how to extend torsors if they are related to other torsors that we know to be extensible. The proof of lemma \ref{lemmaVarAb} is similar to that of theorem \ref{teo2}, so we only sketch it.

\begin{lemma}\label{lemmaVarAb} Let $(Y,G,y)\in \mathcal{P}(X)$  and $(Y_{\eta},G_{\eta},y_{\eta})$ its generic fiber. Let $H'$ be a $K$-group, $u:G_{\eta}\to H'$ a faithfully flat morphism and $(Z',H',z')$  the associated object of $\mathcal{P}(X_{\eta})$. Then there exists a triple $(Z,H,z)\in \mathcal{P}(X)$ whose generic fiber is isomorphic to $(Z',H',z')$.\end{lemma} \proof Set $N:=ker(u)$, then construct the scheme theoretic closure $\overline{N}$ of $N$ in $G$ and consider the quotient $H:=G/N$ in order to have the following diagram:
 $$\xymatrix{N\ar[r] \ar@{^{(}->}[d]& \overline{N} \ar@{^{(}->}[d]\\
 {G}_{\eta}\ar@{->>}[d]_{u}\ar[r] & {G}\ar@{->>}[d]\\
 H'\ar[r]\ar[d] & H\ar[d]\\
 {\eta}\ar[r] & S}$$
\endproof

Roughly speaking this means that if we are able to extend a triple then we can extend any triple which is a ``quotient'' of the previous one (i.e. the morphism between their group schemes is faithfully flat). In the following lemma we solve a similar problem: suppose we are able to extend a triple, then we want to know whether we can extend a triple that ``contains'' it. More precisely:

\begin{lemma}\label{lemmaVarAb2} Let $(Y,G,y)\in \mathcal{P}(X)$  and $(Y_{\eta},G_{\eta},y_{\eta})$ its generic fiber. Let $H'$ be a $K$-group scheme such that $u:G_{\eta}\hookrightarrow H'$ is a closed immersion and let $(Z',H',z')$ be the associated triple of $\mathcal{P}(X_{\eta})$. Then there exists a triple $(Z,H,z)\in \mathcal{P}(X)$ whose generic fiber is isomorphic to $(Z',H',z')$ if and only if there exists a group scheme $L$ finite and flat over $S$ whose generic fiber is isomorphic to $H'$.\end{lemma} \proof If the triple $(Z,H,z)\in \mathcal{P}(X)$ exists just set $L:=H$. The other direction is the non abelian version of \cite{Ray}, Proposition 2.3.1 (a): we are in the following situation
$$\xymatrix{{G}_{\eta}\ar@{^{(}->}[d]_{u}\ar[r] & {G}\\
 H'\ar[r]\ar[d] & L\ar[d]\\
 {\eta}\ar[r] & S}$$
and we want to construct a $S$-finite and flat group scheme $H$ and a morphism $v:G\to H$ in order to obtain a cartesian diagram:
$$\xymatrix{{G}_{\eta}\ar@{^{(}->}[d]_{u}\ar[r] & {G}\ar@{-->}[d]^{v} \\
 H'\ar[r]\ar[d] & H\ar[d]\\
 {\eta}\ar[r] & S.}$$
We can assume $S$ to be affine so set $S:=Spec(R_S)$, $G:=Spec(A)$, $L:=Spec(B)$, $G_{\eta}:=Spec(A_{\eta})$ and $H':=Spec(B_{\eta})$ and consider the induced diagram
$$\xymatrix{{A}_{\eta} & {A}\ar[l] \\
 B_{\eta} \ar@{->>}[u]^{\hat{u}} & B\ar[l]\\
 {K}\ar[u] & R_S\ar[u]\ar[l]}$$ 
where  $\eta:=Spec(K)$. Now let $A^{\vee}$ and $B^{\vee}$ be the duals respectively of the commutative (but not necessarily cocommutative) Hopf $R_S$-algebras $A$ and $B$: they are cocommutative (but not necessarily commutative and not necessarily Hopf) $R_S$-bialgebras; the antipodal morphisms $S_A:A\to A$, $S_{B}:B\to B$, which are morphisms of $R_S$-algebras, are transformed into $R_S$-coalgebra morphisms $S_{A}^{\vee}:{A^{\vee}}\to {A^{\vee}}$, $S_{B}^{\vee}:B^{\vee}\to B^{\vee}$. Similarly let $A_{\eta}^{\vee}$  and $B_{\eta}^{\vee}$ be the cocommutative $K$-bialgebras, duals respectively of $A_{\eta}$  and $B_{\eta}$, provided with the $K$-coalgebra morphisms $S_{A_{\eta}}^{\vee}:{A_{\eta}^{\vee}}\to {A_{\eta}^{\vee}}$, $S_{B_{\eta}}^{\vee}:B_{\eta}^{\vee}\to B_{\eta}^{\vee}$, then consider the diagram

$$\xymatrix{
 {K}\ar[d] & R_S\ar[d]\ar[l]\\{A}_{\eta}^{\vee} \ar[d]_{\hat{u}^{\vee}}& {A}^{\vee}\ar[l] \\
 B_{\eta}^{\vee}  & B^{\vee}.\ar[l]}$$ 
 Now factor $\hat{u}^{\vee}=\varphi\circ \rho_{\eta} :{A}_{\eta}^{\vee}\to {B}_{\eta}^{\vee}$ where $\rho_{\eta}: {A}_{\eta}^{\vee}\to {A}_{\eta}^{\vee}\otimes {B}_{\eta}^{\vee}$ is the generic fiber of $\rho: {A}^{\vee}\to {A}^{\vee}\otimes {B}^{\vee}$, $x\mapsto x\otimes 1$ and $\varphi:{A}_{\eta}^{\vee}\otimes {B}_{\eta}^{\vee}\twoheadrightarrow {B}_{\eta}^{\vee}$, $x\otimes y \mapsto \hat{u}^{\vee}(x)\cdot y$ and consider the diagram 
 $$\xymatrix{
 {K}\ar[d] & R_S\ar[d]\ar[l]\\{A}_{\eta}^{\vee} \ar[d]_{\rho_{\eta}}& {A}^{\vee}\ar[d]^{\rho}\ar[l] \\{A}_{\eta}^{\vee}\otimes {B}_{\eta}^{\vee}\ar@{->>}[d]_{\varphi} & {A}^{\vee}\otimes {B}^{\vee}\ar[l]\\
 B_{\eta}^{\vee}.  & }$$
 According to \cite{EGAIV-2} (Proposition 2.8.1; (2.8.3); Proposition 2.8.4) we complete the previous diagram  by constructing $\varphi':{A}^{\vee}\otimes {B}^{\vee}\twoheadrightarrow C$, where $C$ is the only $R_S$-flat module quotient of ${A}^{\vee}\otimes {B}^{\vee}$ whose generic fiber is isomorphic to ${B}_{\eta}^{\vee}$ (it is moreover a  cocommutative bialgebra provided with a $R_S$-coalgebra morphism $S_C:C\to C$). Its dual $C^{\vee}$ is a commutative bialgebra, flat over $R_S$ whose generic fiber is isomorphic to $B_{\eta}$, and the dual morphism $S_{C}^{\vee}:C^{\vee}\to C^{\vee}$ gives $C^{\vee}$ a  Hopf algebra structure;  set $H:=Spec(C^{\vee})$, consider $\varphi'^{\vee}:C^{\vee}\to {A}\otimes {B}$ the dual morphism of $\varphi'$ then the composition $\rho^{\vee}\circ \varphi'^{\vee}: C^{\vee}\to A$ induces a morphism of $S$-group schemes $\psi: G\to H$ which allows us to construct the desired triple $(Z,H,z)$ as the contracted product of $(Y,G,y)$ via $\psi$.
\endproof

\begin{remark} The assumption that $u:G_{\eta}\to H'$ is a closed immersion is never used, but it is the only case of interest in our situation according to lemma \ref{lemFATTORIZZAGRUPPI}.
\end{remark}

\begin{remark} When $char(K)=0$ a group scheme $H$ as in the statement of Lemma \ref{lemmaVarAb2} always exists.
\end{remark}

\subsection{The case of an abelian scheme}
\label{sez:abel}

Let $S$ be any connected Dedekind scheme and assume that $X\to S$ is an abelian scheme (i.e. $X$ is a smooth and proper  $S$-group scheme with geometric connected fibers), let $0_X$ be the unity for the group law of $X$ and for any natural number $m$ let $m_X:X\to X$ denote the multiplication by $m$. One observes that $(X,{}_mX,0_X)$ is a triple over $X$ where ${}_mX:=ker(m_X)$. In \cite{No3} Nori proves that for any point $s\in S$ and for any triple $(Y',G',y')\in \mathcal{P}(X_{s})$  there exists a natural number $n$ and 
\begin{enumerate}
	\item a morphism of group schemes $u:{}_n(X_{s})\to G'$, 
	\item a morphism $X_{s}\to Y'$ commuting with the actions of ${}_n(X_{s})$ and $G'$
\end{enumerate}

hence in particular, $\pi_1(X_{s},0_{X_{s}})\simeq \underleftarrow{lim}_{n} ({}_n(X_{s}))$. It is clear that the triple $(X_{s},{}_n(X_{s}),0_{X_{s}})$ is isomorphic to the fiber in $s$ of $(X,{}_nX,0_{X})$ then we have the following 

\begin{proposition}\label{PropVarAb} Let $X$ be an abelian scheme over a connected Dedekind scheme $S$, then for every $s\in S$ the canonical morphism $$\varphi:\pi_{1}(X_{s},x_{s})\to
\pi_{1}(X,x)_{s}$$ is an isomorphism.  Moreover $\pi_1(X,0_X)\simeq \underleftarrow{lim}_{n} ({}_nX)$.
\end{proposition}
\proof Every triple $(Y_s,G_s,y_s)$ over $X_s$ fiber of a triple $(Y,G,y)$ over $X$ is preceded by a triple $(X_{s},{}_n(X_{s}),0_{X_{s}})$ for a certain $n$ then every arrow $\pi_{1}(X,0_X)_{s}\to G_s$ factors through ${}_n(X_{s})$ hence the first assertion. Thus for any $s\in S$ the fiber $\psi_s:\pi_1(X,0_X)_s\to \underleftarrow{lim}_{n} ({}_n(X_s))$ is an isomorphism. As usual let $\eta$ be the generic point of $S$ and consider the canonical morphism $\psi:\pi_1(X,0_X)\to \underleftarrow{lim}_{n} ({}_n X)$. Assume for a moment that it is faithfully flat,  then it is sufficient to consider its kernel which is flat and then trivial since it coincides with the scheme theoretic closure of $\{1\}_{\eta}$, the kernel of $\psi_{\eta}$, in $\pi_1(X,0_X)$. In order to prove it is faithfully flat  then one has to show that for any $n\in \mathbb{N}$ the canonical morphism $\pi_1(X,0_X)\to {}_n X$ is faithfully flat too. To show this, considering the usual projective limit $\pi_{1}(X,0_X)=\underleftarrow{lim}_{i\in I} G_i$, it is sufficient to prove that whenever there is a morphism $G_i\to {}_n X$ then it is faithfully flat. But this is certainly true on the fibers, since $(X_{s},{}_n(X_{s}),0_{X_{s}})$ is a dominated triple for any $s\in S$ then $(G_i)_s\to {}_n (X_s)$ is faithfully flat. Thus $G_i\to {}_n X$ is faithfully flat by means of \cite{EGAIV-3} Théorème 11.3.10. Hence $\psi$ is an isomorphism according to \cite{DG}, I, \S 2, n$^{\circ}$ 3, Corollaire 2.9. \endproof

\begin{remark} That $\pi_1(X,0_X)$ is isomorphic to $\underleftarrow{lim}_{n} ({}_n X)$ can be proven directly without considering fibers, following what Nori did in \cite{No3}. We cannot use \cite{EGAIV-3} Théorème 11.3.10 for $\psi$ because in general the fundamental group scheme is not of finite type.
\end{remark}

Let $\eta$ be the generic point of $S$,  if $(Y',G',y')$ is a dominated triple over $X$ then there exists $n\geq 1$ such that $u:{}_n(X_{s})\to G'$ is faithfully flat and then, according to lemma \ref{lemmaVarAb}, there exists a triple $(Y,G,y)$ extending $(Y',G',y')$. Then we have the following

\begin{proposition}\label{PropVarAb2} When $X$ is an abelian scheme over a connected Dedekind scheme $S$, then every dominated triple $(Y',G',y')\in \mathcal{P}(X_{\eta})$ can be extended to a triple $(Y,G,y)\in  \mathcal{P}(X)$; moreover every non dominated triple $(Y',G',y')$ can be extended to a triple $(Y,G,y)$ if and only if there exists a finite and flat $S$-group scheme $H$ whose generic fiber is isomorphic to $G'$. 
\end{proposition}
\proof That every dominated triple can be extended follows from previous discussion. Then apply lemma \ref{lemmaVarAb2} to obtain the statement on non dominated triples.\endproof

\textbf{Acknowledgements}. This paper is part
of my PhD thesis. I would like to thank my advisor Michel Emsalem
for his guidance and his constant encouragement. This work has
been partially supported by the Università degli Studi di
Milano.  I also would like to thank Matthieu Romagny, Carlo
Gasbarri and Dajano Tossici for useful comments and discussions. Finally I would like to thank an anonymous referee for his suggestions, corrections and some improvements.

\begin{flushright}Marco Antei\\ Laboratoire Paul Painlevé, U.F.R. de
Mathématiques\\Université des Sciences et des Techonlogies de
Lille 1\\ 59 655 Villeneuve d'Ascq\\\end{flushright}

\begin{flushright} E-mail:\\ \texttt{antei@math.univ-lille1.fr}\\ \texttt{marco.antei@gmail.com}
\end{flushright}

\end{document}